\documentclass[12pt, reqno]{amsart}
 \usepackage{amsmath, amsthm, amscd, amsfonts, amssymb, graphicx, color, float}
\usepackage[bookmarksnumbered, colorlinks, plainpages]{hyperref}
\input{mathrsfs.sty}
\hypersetup{colorlinks=true,linkcolor=red, anchorcolor=green, citecolor=cyan, urlcolor=red, filecolor=magenta, pdftoolbar=true}

\newtheorem{theorem}{Theorem}[section]
\newtheorem{lemma}[theorem]{Lemma}

\newtheorem{corollary}[theorem]{Corollary}
\theoremstyle{definition}

\newtheorem{example}[theorem]{Example}

\theoremstyle{remark}

\numberwithin{equation}{section}

\begin{document}

\title[Pedersen--Takesaki operator equation]{Pedersen--Takesaki operator equation in Hilbert $C^*$-modules}

\author[R. Eskandari, X. Fang, M.S. Moslehian, Q. Xu]{Rasoul Eskandari$^{1}$, Xiaochun Fang$^2$, M. S. Moslehian$^3$, Qingxiang Xu$^4$}
\address{$^1$Department of Mathematics, Faculty of Science, Farhangian University, Tehran, Iran.}
\email{eskandarirasoul@yahoo.com}

\address{$^2$Department of Mathematics, Tongji University, Shanghai 200092, PR China.}
\email{xfang@tongji.edu.cn}

\address{$^3$Department of Pure Mathematics, Center of Excellence in
Analysis on Algebraic Structures (CEAAS), Ferdowsi University of Mashhad, P. O. Box 1159, Mashhad 91775, Iran.}
\email{moslehian@um.ac.ir}

\address{$^4$Department of Mathematics, Shanghai Normal University, Shanghai 200234, PR China}
\email{qingxiang\_xu@126.com}

\renewcommand{\subjclassname}{\textup{2020} Mathematics Subject Classification} 


\begin{abstract}
We extend a work of Pedersen and Takesaki by giving some equivalent conditions for the existence of a positive solution of the so-called Pedersen--Takesaki  operator equation $XHX=K$ in the setting of Hilbert $C^*$-modules. It is known that the Douglas lemma does not hold in the setting of Hilbert $C^*$-modules in its general form. In fact, if $\mathscr{E}$ is a Hilbert $C^*$-module and  $A, B \in  \mathcal{L}(\mathscr E)$, then the operator inequality $B B^*\le \lambda AA^*$ with $\lambda>0$ does not ensure that the operator equation $AX=B$ has a  solution, in general. We show that under a mild orthogonally complemented condition on the range of operators,  $AX=B$ has a solution if and only if $BB^*\leq \lambda AA^*$ and $\mathscr R(A) \supseteq \mathscr R(BB^*)$.  Furthermore, we prove that if $\mathcal{L}(\mathscr E)$ is a $W^*$-algebra,  $A,B\in \mathcal{L}(\mathscr E)$, and $\overline{\mathscr R(A^*)}=\mathscr E$, then $BB^*\leq\lambda AA^*$ for some $\lambda>0$ if and only if $\mathscr R (B)\subseteq \mathscr R(A)$. Several examples are given to support the new findings.
\end{abstract}

\maketitle
\section{Introduction}
A pre-Hilbert $C^*$-module over a $C^*$-algebra $\mathscr A$ is a right $\mathscr A$-module equipped with an $\mathscr A$-valued inner product $\langle \cdot,\cdot\rangle:\mathscr H\times \mathscr H\to \mathscr A$. If $\mathscr H$ is complete with respect to the induced norm $\|x\|=\|\langle x,x\rangle\|^{\frac{1}{2}}$, then $\mathscr H$ is called a Hilbert $C^*$-module.  Let $\mathscr R(A)$ and $\mathscr N(A)$ denote the range and the null space of a linear operator $A$, respectively.  Throughout this note, a capital letter means a linear operator. Let $\mathcal{L}(\mathscr H,\mathscr K)$ stand for the set of operators  $A:\mathscr H\to\mathscr K$ for which there is an operator $A^*:\mathscr K\to \mathscr H$ such that
\[
\langle Ax,y\rangle=\langle x,A^*y\rangle\,,\quad x\in \mathscr H,y\in \mathscr K.
\]
It is known that not every bounded linear operator is adjointable. Let $\mathscr E$ be a Hilbert  $\mathscr A$-module. For every subset $\mathscr S$ of $\mathscr E$, we define $\mathscr S^\perp$ by
\[
\mathscr S^\perp:=\{x\in \mathscr E; \langle x,y\rangle=0,{\rm ~for~all~} y\in \mathscr S\}.
\]
Let $\mathscr M$ be a closed submodule of $\mathscr E$. If $\mathscr E=\mathscr M\oplus \mathscr M^\perp$, then we say that $\mathscr M$ is orthogonally complemented in $\mathscr E$. There are Hilbert $C^*$-modules having a noncomplemented closed submodule; see \cite[Chapter 3]{Lan}. Furthermore, when the range of $A\in \mathcal{L}( \mathscr{H})$ is not closed, there may happen that neither $ {\mathcal N}(A) $ nor $ \overline{{\mathcal R}(A)} $ is orthogonally complemented.  Indeed, the theory Hilbert $C^*$-modules is a nontrivial extension of that of Hilbert spaces utilizing complecated operator algebra techniques; see \cite{man2}.

For an operator $A\in \mathcal{L}(\mathscr F,\mathscr E)$, if the closure of $\mathscr R(A^*)$ is orthogonality complemented, then $\overline{\mathscr R(A^*)}^\perp=\mathscr N(A)$ and $\mathscr F=\overline{\mathscr R(A^*)}^\perp\oplus\mathscr N(A)$. Let $A\in \mathcal{L}(\mathscr F,\mathscr E)$. If $\overline{\mathscr R(A^*)}$ is an orthogonally complemented submodule in $\mathscr F$, then we denote by $N_A$ the projection $I-P_{A^*}$, where $P_{A^*}$ is the projection  $\mathscr F$ onto $\overline{\mathscr R(A^*)}$. Also from \cite[Proposition 3.7]{Lan} (see also \cite{VMX}), we have 
\begin{equation}\label{eqclosed}
\overline{\mathscr R(A)}=\overline{\mathscr R(AA^*)}\,.
\end{equation}
The solvability of operator equations in Hilbert $C^*$-module is one of the important topics in the recent researches. In \cite{XTA}, it was given a necessary  and sufficient condition for the existence of a solution to the equation $AXB=C$.
\begin{theorem}\cite[Theorem 3.1]{XTA}\label{t1}
 Let  $A\in \mathcal{L}(\mathscr E,\mathscr  F),B \in \mathcal{L}(\mathscr G,\mathscr H)$ and $C\in \mathcal{L}(\mathscr G,\mathscr  F)$.
 \begin{itemize}

\item[(i)] If the equation $AXB = C$ has a solution $X \in \mathcal{L}(\mathscr H,\mathscr E)$, then
\[
\mathscr R(C)\subseteq \mathscr R(A)~ and ~\mathscr R(C^*)\subseteq \mathscr  R(B^*).
\]
\item[(ii)] Suppose $\overline{\mathscr R(B)}$ and $\overline{\mathscr R(A^*)}$ are orthogonally complemented submodules of
$\mathscr H$ and $\mathscr E$, respectively. If
\[
\mathscr R(C) \subseteq  \mathscr R(A) ~and~\overline{\mathscr  R(C^*)} \subseteq \mathscr R(B^*)
\]
\[
(or, \overline{\mathscr R(C)} \subseteq \mathscr R(A)~ and ~\mathscr R(C^*) \subseteq \mathscr R(B^*)),
\]
then $AXB = C$ has a unique solution $D\in \mathcal{L}(\mathscr H,\mathscr E)$ such that
\[
\mathscr R(D) \subseteq \mathscr  N(A)^\perp~ and ~\mathscr R(D^*) \subseteq \mathscr N(B^*)^\perp,
\]
which is called the reduced solution, and the general solution to $AXB =C$ is of the form
\[
X = D + N_AV_1 + V_2N_{B^*} ,\quad  where ~V_1, V_2\in \mathcal L(\mathscr H,\mathscr E).
\]
Moreover if at least one of $\mathscr R(A), \mathscr R(B)$, and $\mathscr R(C)$ is a closed submodule,
then the condition $\overline{\mathscr R(C^*)} \subseteq \mathscr R(B^*)$ (or $\overline{\mathscr R(C)}\subseteq \mathscr R(A)$) can be
replaced by $\mathscr R(C^*) \subseteq \mathscr  R(B^*)$ (or $\mathscr R(C) \subseteq \mathscr R(A)$).
\end{itemize}
\end{theorem}

The quadratic equation $XA^{-1}X=B$ for positive invertible operators $A$ and $B$ is called the Riccati equation, and its unique positive solution is the geometric mean $A\#B=A^{1/2}(A^{-1/2}BA^{-1/2})^{1/2}A^{1/2}$ of $A$ and $B$; see \cite{and}. 

The Riccati type operator equation $THT=K$ was considered by Pedersen and Takesaki \cite{Ped}. They \cite{Ped2} observed that it is a useful tool for the study of  noncommutative Radon--Nikod\'{y}m theorem. They proved the following interesting result.
\begin{theorem}\cite{Ped}.\label{t4}
Let $\mathscr H$ be a Hilbert space. Let $ H$ and $ K$ be positive operators in the $C^*$-algebra $\mathcal{B}(\mathscr H)$ of all bounded linear operators, and assume that $H$ is invertible. There is then at most one positive operator $X$ in $\mathcal{B}(\mathscr H)$ such that $XHX=K$. A necessary and sufficient condition for the existence of such $X$ is that $(H^{\frac{1}{2}}KH^{\frac{1}{2}})^{\frac{1}{2}}\leq aH$ for some $a>0$, and then $\|X\|\leq a$. This condition will be satisfied if $H$ is invertible or, more generally, if  $K\leq a^2H$ for some $a>0$.
\end{theorem}
Nakamoto \cite{Nak} demonstrated the connection of the equation $XHX=K$ with the Douglas lemma  \cite[Theorem 1]{Dou} when $ H$ and $ K$ are positive operators in $\mathcal{B}(\mathscr H)$; see also \cite{cve, man}.
Furuta \cite{Fur} extended the main result of  Pedersen and Takesaki and showed that if $H, K \in \mathcal{B}(\mathscr H)$ are positive operators, $H$ is nonsingular, and  there exists a positive operator $X$ such that $X(H^{1/n}X)^n=K$ for some natural number $n$, then, for any natural number $m$ such that $m\leq n$, there exists a positive operator $X_1$ such that $X_1(H^{1/m}X_1)^m=K$. 
 
An extension of the Pedersen--Takesaki operator equation has been investigated via the Furuta inequality and the grand Furuta inequality by Yuan and Gao \cite{yg}.  In addition, Shi and Gao \cite{sg}  presented some applications of the Furuta inequality to some Pedersen--Takesaki type operator equations. One of our aims is to extend Theorem \ref{t4} to the setting of Hilbert $C^*$-modules and to provide  some equivalent conditions for the existence of a positive solution of $XHX=K$.

It is known that the Douglas lemma  does not hold in the setting of Hilbert $C^*$-modules in its general form. In fact, for $A, B \in  \mathcal{L}(\mathscr E)$, the operator inequality $B B^*\le \lambda AA^*$ with $\lambda>0$ does not imply that $AX=B$ has a solution, in general; for more information, see \cite{man}. In the present paper, we show that under a simple orthogonally complemented condition on the range of operators,  $AX=B$ has a solution if and only if $BB^*\leq \lambda AA^*$ and $\mathscr R(A) \supseteq \mathscr R(BB^*)$. Furthermore, we prove that if $\mathcal{L}(\mathscr E)$ is a $W^*$-algebra,  $A,B\in \mathcal{L}(\mathscr E)$, and $\overline{\mathscr R(A^*)}=\mathscr E$, then $BB^*\leq\lambda AA^*$ for some $\lambda>0$ if and only if $\mathscr R (B)\subseteq \mathscr R(A)$.

\section{Main Results}

If $A\in \mathcal{L}(\mathscr E) $ has closed range, then the Moore--Penrose inverse of $ A$ \cite{Q} is defined as the unique element $A^\dagger$ of $\mathcal{L}(\mathscr E)$ such that
\[
AA^\dagger A=A\,,\quad A^\dagger AA^\dagger=A^\dagger\,,\quad (A^\dagger A)^*= A^\dagger A\,, \quad (AA^\dagger )^* = AA^\dagger.
\]
We know that $A^\dagger $ exists if and only if $\mathscr R(A)$ is closed. Note that the Moore-Penrose inverse can also be dealt with for some operators whose ranges may be not closed. More precisely, suppose that $A\in \mathcal{L}(\mathscr F)$ and  $\overline{\mathscr R(A^*)}$ is orthogonally complemented in $\mathscr F$. From \cite[Remark 1.1]{AOFA}, the Moore--Penrose inverse $A^\dagger $
can be constructed as the unique operator from $\mathscr R(A)\oplus \mathscr  N(A^*)$ onto $\mathscr R(A^*)$ such
that for every $x\in \mathscr  R(A^*)$ and $y\in \mathscr  N(A^*)$,
\[
A^\dagger Ax = x\,, \quad  A^\dagger y = 0.
\]
We have  $AA^\dagger A = A$ and $A^\dagger AA^\dagger  = A^\dagger$. Note that  $A^\dagger $ may be unbounded (thus $A^\dagger $
may fail to be in $\mathcal{L}(\mathscr E ,\mathscr F )$) and that the operator $A^\dagger A$ is however always bounded, which
is actually equal to the projection from $\mathscr E$ onto $\overline{\mathscr R(A^*)}$.\\
 In the Hilbert space setting,  Douglas lemma gives a necessary and sufficient conditions for solvability of the equation $AX=B$ as follows.
 \begin{theorem}\cite[Theorem 1]{Dou}\label{Douglas}
 Let $A$ and $B$ be bounded linear operators on the Hilbert space $\mathscr H$. The following statements are equivalent:
 \begin{itemize}
 \item[(1)] $\mathscr R(A)\subseteq \mathscr R(B)$.
 \item[(2)] $AA^*\leq \lambda BB^* $ for some $\lambda>0 $.
 \item[(3)] there exists a bounded operator $C$ on $\mathscr H$ such that $A=BC$.
 \end{itemize}
 \end{theorem} 
 In the Hilbert $C^*$-module setting, we have the following theorem.
\begin{theorem}\cite[Theorem 2.3]{MR}\label{thFang}
Let $C\in \mathcal{L}(\mathscr G,\mathscr  F), A \in \mathcal{L}(\mathscr E,\mathscr  F)$. Suppose that $\overline{\mathscr R(A^*)}$ is orthogonally complemented. The following statements are equivalent:
\begin{itemize}
\item[(1)]$\mathscr R(C)\subseteq \mathscr R(A)$.
\item[(2)] There exists $D\in \mathcal{L}(\mathscr G, \mathscr E)$ such that $AD=C$, i.e., $AX=C$ has a solution $X$.
\end{itemize}
And if one of the above statements holds, then
$$CC^*\leq  \lambda AA^*,\qquad {\rm for~ some~} \lambda > 0.$$
In this case, there exists a unique operator $X$ satisfying $\mathscr R(X)\subseteq \mathscr N(A)^\perp$, which is called the reduced solution, denoted by $D$, and is defined as follows:
$$
D = P_{A^*}A^{-1}C,\quad  D^*|_{\mathscr N(A)} = 0, \quad D^*A^*y = C^*y\quad  (y\in \mathscr  F),
$$
where $A^{-1}$ does not refer to the inverse of $A$ but is the expression of inverse image.
\end{theorem}
\begin{theorem}\cite[Theorem 2.1]{AOFA}\label{RA}
 Let $A\in \mathcal  L(\mathscr H ,\mathscr K ), B\in \mathcal L(\mathscr L,\mathscr  G )$  be such that $\overline{R(A^*)}$ and $\overline R(B)$
are orthogonally complemented. Then for every $C\in \mathcal L(\mathscr L, \mathscr K )$, the equation
\begin{equation}\label{eqp}
AXB = C, \qquad X\in\mathcal  L(\mathscr G ,\mathscr H ) 
\end{equation}
has a solution if and only if
\[
\mathscr R(C)\subseteq \mathscr R(A)  \qquad \textrm{and}\qquad \mathscr R\left(A^\dagger C)^*\right)\subseteq \mathscr R(B^*).
\]
\end{theorem}
\begin{theorem}
Let $A\in \mathcal{L}(\mathscr E,\mathscr F)$ and let $C\in \mathcal{L}(\mathscr F)$. Let $\overline{\mathscr R(A^*)}$ be an orthogonally complemented closed submodule of $\mathscr E$. Then the following assertions are equivalent:
\begin{itemize}
\item[(i)] $AXA^*=C$ has a positive solution $X$.
\item[(ii)] $\mathscr R(C)\subseteq \mathscr R(A), \mathscr R(A^\dagger C)^*\subseteq \mathscr R(A)$, and $C\geq 0$.
\end{itemize}
\end{theorem}
\begin{proof}
(i)$\Rightarrow$(ii). In virtue of  \cite[Theorem 2.1]{AOFA} it is  clear. \\
(ii)$\Rightarrow$(i). It follows from  $\mathscr R (C)\subseteq \mathscr R(A)$, $\mathscr R(A^\dagger C)^*\subseteq \mathscr R(A)$, and  \cite[Theorem 2.1]{AOFA} that  $X=(A^\dagger (A^\dagger C)^*)^*$ is a solution to  the equation $AXA^*=C$. We observe that $\mathscr R(X^*)\subseteq \mathscr N(A)^\perp$ and $\mathscr R(X)\subseteq \mathscr N(A)^\perp$. Indeed, for every $x\in \mathscr E$ and every  $z\in \mathscr N(A)$, we observe that there exists $y\in\mathscr E$ such that
\begin{align*}
\langle X^*x,z\rangle &=\langle (A^\dagger (A^\dagger C)^*)x,z\rangle=\langle A^\dagger Ay,z\rangle=0.
\end{align*}
Moreover, we have
\begin{align*}
\langle Xx,z\rangle &=\langle (A^\dagger (A^\dagger C)^*)^*x,z\rangle\\\nonumber
&=\langle x,A^\dagger (A^\dagger C)^*(I-A^\dagger A)z\rangle \qquad(\rm{since}~z\in \mathscr R(I-A^\dagger A))\\\nonumber
&=\langle x,(A^\dagger (A^\dagger C)^*-A^\dagger (A^\dagger C)^*(A^\dagger A))z\rangle\\\nonumber
&=\langle x,(A^\dagger (A^\dagger C)^*-A^\dagger (A^\dagger AA^\dagger C)^*)z\rangle\\\nonumber
&=\langle x,(A^\dagger (A^\dagger C)^*-A^\dagger (A^\dagger C)^*)z\rangle =0.\\\nonumber
\end{align*}
Now we claim that $X\geq 0$. To show this, let $x\in \mathscr R(A^*)$. Then for some $y\in \mathscr F$, we have
\begin{eqnarray}\label{po}
\langle Xx,x\rangle =\langle XA^*y,A^*y\rangle = \langle AXA^*y,y\rangle =\langle Cy,y\rangle\geq 0\,.
\end{eqnarray}
Since $\mathscr R(X)\subseteq \mathscr N(A)^\perp$ and $\mathscr R(X^*)\subseteq \mathscr N(A)^\perp$  for any $z\in \mathscr N(A)$ and $x\in \overline{\mathscr R(A^*)}$, we have
\begin{align}\label{po2}
\langle Xz,z\rangle =0 \qquad\textrm{and}\qquad  \langle X^*x,z\rangle =0\,.
\end{align}
Let $x\in \mathscr E$.
Since $\mathscr E=\overline{\mathscr R(A^*)}\oplus \mathscr N(A)$, we can write $x=x_1+x_2$ with $x_1\in \overline{\mathscr R(A^*)}$ and $x_2\in \mathscr N(A)$. Hence
\begin{align*}
\langle Xx,x\rangle &= \langle X(x_1+x_2),x_1+x_2\rangle\\
&=\langle Xx_1,x_1\rangle+\langle Xx_1,x_2\rangle+\langle Xx_2,x_2\rangle+\langle Xx_2,x_1\rangle\\
&=\langle Xx_1,x_1\rangle+\langle x_2,X^*x_1\rangle\qquad(\rm by \eqref{po2})\\
&=\langle Xx_1,x_1\rangle\geq 0\qquad (\rm by \eqref{po}).
\end{align*}
\end{proof}
We say that a self-adjoint operator $H\in \mathcal{L}(\mathscr E)$ is nonsingular if $\overline{\mathscr R(H)}=\mathscr E$. In this case, we have $\mathscr N(H)=\{0\}$. Note that if $H$ is a positive nonsingular operator, then,  \eqref{eqclosed} implies that
\[
\overline{\mathscr R(H^{\frac{1}{2}})}=\overline{\mathscr R(H)}=\mathscr E\,.
\]
This shows $H^{\frac{1}{2}}$ is also  nonsingular.
\begin{theorem}\label{t3}
Let $H,K\in \mathcal{L}(\mathscr E)$ with $H\geq 0$. Let $\overline{\mathscr R(H)}$ be an orthogonally complemented closed submodule of $\mathscr E$. If
the equation
\begin{equation}\label{eq1}
XHX=K
\end{equation}
has a solution $X$ in $\mathcal{L}(\mathscr E)$, then
\begin{equation}\label{eq2}
\mathscr R(H^{\frac{1}{2}}KH^{\frac{1}{2}})^{\frac{1}{2}}\subseteq \mathscr R(H^{\frac{1}{2}})\,,\qquad\mathscr R(H^{\frac{1}{2}\dagger}(H^{\frac{1}{2}}KH^{\frac{1}{2}})^{\frac{1}{2}})^*\subseteq \mathscr R(H^{\frac{1}{2}}).
\end{equation}
If $H$ is nonsingular, then \eqref{eq2} is sufficient for solvability of \eqref{eq1}, so in this case  \eqref{eq1} has a unique solution.
\end{theorem}
\begin{proof}
Suppose that \eqref{eq1} is solvable. Let $T\in \mathcal L(\mathscr E)$ be such that $THT=K$. Then  
\[
\left(H^{\frac{1}{2}}TH^{\frac{1}{2}}\right)^2=H^{\frac{1}{2}}THTH^{\frac{1}{2}}=H^{\frac{1}{2}}KH^{\frac{1}{2}}\,.
\]
Hence $H^{\frac{1}{2}}TH^{\frac{1}{2}}=\left(H^{\frac{1}{2}}KH^{\frac{1}{2}}\right)^{\frac{1}{2}}$. By setting $A=B=H^{\frac{1}{2}}$ and $C=\left(H^{\frac{1}{2}}KH^{\frac{1}{2}}\right)^{\frac{1}{2}}$ and employing  Theorem \ref{t1}, we get 
\begin{align}
T=D+ N_{A}V_1 + V_2N_{A} ,
\end{align}
where  $V_1, V_2\in \mathcal L(\mathscr E)$ and $D$ is the reduced solution. Since $H$ is nonsingular, $\overline{\mathscr R(H^{\frac{1}{2}})}=\mathscr E$. Thus $N_A=N_{H^{\frac{1}{2}}}=0$. Hence $T=D$. This shows that $XHX=K$ has at most one solution.\\
Let $T\in \mathcal{L}(\mathscr E)$ be a solution to \eqref{eq1}. Then  
\[
H^{\frac{1}{2}}KH^{\frac{1}{2}}=H^{\frac{1}{2}}THTH^{\frac{1}{2}}=(H^{\frac{1}{2}}TH^{\frac{1}{2}})^2.
\]
This shows that
\[
H^{\frac{1}{2}}TH^{\frac{1}{2}}=(H^{\frac{1}{2}}KH^{\frac{1}{2}})^{\frac{1}{2}}.
\]
Hence, Theorem \ref{RA} yields \eqref{eq2}.\\
 Now, let $H$ be nonsingular and let \eqref{eq2} hold.
Making use of Theorem \ref{RA}, we derive that there is $T\in \mathcal{L}(\mathscr E)$ such that
\[
H^\frac{1}{2}TH^{\frac{1}{2}}=(H^{\frac{1}{2}}KH^{\frac{1}{2}})^{\frac{1}{2}}.
\]
Therefore, we get
\[
H^{\frac{1}{2}}THTH^{\frac{1}{2}}=H^{\frac{1}{2}}KH^{\frac{1}{2}}.
\]
The nonsingularity of $H$ gives $THT=K$.
\end{proof}

Theorem \ref{t4} provides a necessary and sufficient condition for the existence of a solution to $XHX=K$ in the framework of Hilbert spaces. The next result extends the theorem to the setting of Hilbert $C^*$-modules.
\begin{theorem}\label{t5}
Let $H,K\in \mathcal{L}(\mathscr E)$ be positive operators. Let $H$ be a nonsingular operator and let $\overline{\mathscr R(H^{\frac{1}{2}}KH^{\frac{1}{2}})}$ be an orthogonally complemented submodule of $\mathscr E$. Then the following assertions are equivalent:
\begin{itemize}
\item[(i)] Equation \eqref{eq1} has a positive solution.
\item[(ii)]  $\mathscr R(H^{\frac{1}{2}}KH^{\frac{1}{2}})^{\frac{1}{2}}\subseteq \mathscr R(H^{\frac{1}{2}})$ and $\mathscr R(H^{\frac{1}{2}\dagger}(H^{\frac{1}{2}}KH^{\frac{1}{2}})^{\frac{1}{2}})^*\subseteq \mathscr R(H^{\frac{1}{2}})$.
\item[(iii)] $\mathscr R(H^{\frac{1}{2}}KH^{\frac{1}{2}})^{\frac{1}{4}}\subseteq \mathscr R(H^{\frac{1}{2}})$.
\item[(iv)] $\mathscr R(H^{\frac{1}{2}}KH^{\frac{1}{2}})^{\frac{1}{2}}\subseteq \mathscr R(H^{\frac{1}{2}})$ and $(H^{\frac{1}{2}}KH^{\frac{1}{2}})^{\frac{1}{2}}\leq \lambda H$ for some scalar $\lambda>0$.
\end{itemize}
\end{theorem}
\begin{proof}
(i)$\Longrightarrow$(ii). It is evident from Theorem \ref{t3}.\\
(ii)$\Longrightarrow$(i). As in the proof of Theorem \ref{t3},  there is $T\in \mathcal{L}(\mathscr E)$ such that $THT=K$ and \begin{equation}\label{eq4}
(H^{\frac{1}{2}}KH^{\frac{1}{2}})^{\frac{1}{2}}=H^{\frac{1}{2}}TH^{\frac{1}{2}}\,.
\end{equation}
Multiplying $(H^{\frac{1}{2}})^{\dagger}$ on the both sides of \eqref{eq4} and noting that $(H^{\frac{1}{2}})^{\dagger}H^{\frac{1}{2}}=I$,  we get
\begin{equation}\label{eq5}
(H^{\frac{1}{2}})^{\dagger}(H^{\frac{1}{2}}KH^{\frac{1}{2}})^{\frac{1}{2}}(H^{\frac{1}{2}})^{\dagger}=TH^{\frac{1}{2}}(H^{\frac{1}{2}})^{\dagger}\,.
\end{equation}
 Multiply $H^{\frac{1}{2}}$ on the both sides \eqref{eq5} to get
\[
H^{\frac{1}{2}}\Big[ (H^{\frac{1}{2}})^{\dagger}(H^{\frac{1}{2}}KH^{\frac{1}{2}})^{\frac{1}{2}}(H^{\frac{1}{2}})^{\dagger}\Big] H^{\frac{1}{2}}=H^{\frac{1}{2}}TH^{\frac{1}{2}}\,.
\]
Employing  the nonsingularity of $H$,  we arrive at $T=(H^{\frac{1}{2}})^{\dagger}(H^{\frac{1}{2}}KH^{\frac{1}{2}})^{\frac{1}{2}}(H^{\frac{1}{2}})^{\dagger}$. Now we show that $T$ is positive. It is enough to show that $\langle TH^{\frac{1}{2}}x, H^{\frac{1}{2}}x \rangle\geq 0$ for any $x\in \mathscr E$. Indeed,
\begin{align*}
\langle TH^{\frac{1}{2}}x,H^{\frac{1}{2}}x\rangle &=\langle (H^{\frac{1}{2}})^{\dagger}(H^{\frac{1}{2}}KH^{\frac{1}{2}})^{\frac{1}{2}}(H^{\frac{1}{2}})^{\dagger}H^{\frac{1}{2}}x,H^{\frac{1}{2}}x\rangle\\
&=\langle H^{\frac{1}{2}}(H^{\frac{1}{2}})^{\dagger}(H^{\frac{1}{2}}KH^{\frac{1}{2}})^{\frac{1}{2}}x,x\rangle\\
&=\langle (H^{\frac{1}{2}}KH^{\frac{1}{2}})^{\frac{1}{2}}x,x\rangle\geq 0.
\end{align*}
(i)$\Longrightarrow$(iii). Assume that $THT=K$ for some positive operator $T\in \mathcal{L}(\mathscr E)$. Setting $S=H^{\frac{1}{2}}T^{\frac{1}{2}}$, we get
\begin{equation}\label{eq3}
(H^{\frac{1}{2}}KH^{\frac{1}{2}})^{\frac{1}{2}}=(H^{\frac{1}{2}}THTH^{\frac{1}{2}})^{\frac{1}{2}}=H^{\frac{1}{2}}TH^{\frac{1}{2}}=SS^*\,.
\end{equation}
From \eqref{eqclosed},  we see that
\[
\overline{\mathscr R(H^{\frac{1}{2}}KH^{\frac{1}{2}})}=\overline{\mathscr R(H^{\frac{1}{2}}KH^{\frac{1}{2}})^{\frac{1}{4}}}\,.
\]
Hence $\overline{\mathscr R(H^{\frac{1}{2}}KH^{\frac{1}{2}})^{\frac{1}{4}}}$ is orthogonally complemented. Thus Lemma \ref{l1} and \eqref{eq3} yield that
 \[
 \mathscr R(H^{\frac{1}{2}}KH^{\frac{1}{2}})^{\frac{1}{4}}\subseteq \mathscr R(S)=\mathscr R(H^{\frac{1}{2}}T^{\frac{1}{2}})\subseteq \mathscr R(H^{\frac{1}{2}})\,.
 \]
(iii)$\Longrightarrow$(i). Since $\overline{\mathscr R(H^{\frac{1}{2}})}=\overline{\mathscr R(H)}=\mathscr E$ and $\mathscr R(H^{\frac{1}{2}}KH^{\frac{1}{2}})^{\frac{1}{4}}\subseteq \mathscr R( H^{\frac{1}{2}})$, Theorem \ref{thFang} ensures  the existence of an operator $S\in \mathcal{L}(\mathscr E)$ such that
$H^{\frac{1}{2}}S=(H^{\frac{1}{2}}KH^{\frac{1}{2}})^{\frac{1}{4}}$. Set $T=SS^*$.  Then
\begin{align*}
H^{\frac{1}{2}}THTH^{\frac{1}{2}}&=H^{\frac{1}{2}}SS^*HSS^*H^{\frac{1}{2}}=H^{\frac{1}{2}}KH^{\frac{1}{2}}\,.
\end{align*}
Then nonsingularity of $H$ entails that $THT=K$.\\
(i)$\Longrightarrow$(iv). The validity of (ii) ensures that $\mathscr R(H^{\frac{1}{2}}KH^{\frac{1}{2}})^{\frac{1}{2}}\subseteq \mathscr R(H^{\frac{1}{2}})$. It follows from (iii) that $\mathscr R(H^{\frac{1}{2}}KH^{\frac{1}{2}})^{\frac{1}{4}}\subseteq \mathscr R(H^{\frac{1}{2}})$. Since $\mathscr R(H^{\frac{1}{2}})$ is an orthogonally complemented submodule of $\mathscr E$, by Theorem \ref{thFang}, we get $(H^{\frac{1}{2}}KH^{\frac{1}{2}})^{\frac{1}{2}}\leq \lambda H$ for some $\lambda>0$.  \\
(iv)$\Longrightarrow$(iii). If we set $A=H^{\frac{1}{2}}$ and $B=(H^{\frac{1}{2}}KH^{\frac{1}{2}})^{\frac{1}{4}}$, then we infer from Theorem \ref{t6} that (ii) holds.
\end{proof}

Next, we give some conditions under which $BB^*\leq AA^*$ is equivalent to
$\mathscr R(B)\subseteq \mathscr R(A) $. To achieve it, we need the following lemma, which is a modification of \cite[Theorem 2.4]{MR}. We give its proof for the sake of reader's convenience.

\begin{lemma}\label{l1}
Let $A,C\in \mathcal{L}(\mathscr E)$. Let $\overline{\mathscr R(A^*)}$ be an orthogonally complemented submodule of $\mathscr E$. If $AA^*=\lambda CC^*$ for some $\lambda\in \mathbb{R}$, then $\mathscr R(A)\subseteq \mathscr R(C)$.
\end{lemma}
\begin{proof}
Define $ D:\mathscr R(A^*)\to \mathscr R(C^*)$ by
\[
D(A^*x)=C^*x\qquad (x\in \mathscr E)\,.
\]
Since $AA^*=\lambda CC^*$, we have 
\[
\|A^*x\|=\lambda^{\frac{1}{2}}\|C^*x\|\qquad(x\in \mathscr E ).
\]
Hence $D$ is a well-defined bounded linear operator. Let $D_0$ be the bounded extension of $D$  to $\overline{\mathscr R(A^*)}$. Then
\[
\tilde{D}(y)=\begin{cases}
D_0(y),& y \in \overline{\mathscr R(A^*)},\\
0,& y\in \mathscr  N(A),
\end{cases}
\]
is well-defined and $\tilde{D}A^*=C^*$. In addition,  $C\tilde{D}A^*=CC^*=\lambda AA^*$.

Since $\overline{\mathscr R(A^*)}$ is orthogonally complemented in $\mathscr E$, we get $C\tilde{D}=A$. Hence $\mathscr R(A)\subseteq \mathscr R(C)$.
\end{proof}
\begin{theorem}\label{t6}
Let $\mathscr E$ be a Hilbert $\mathscr{A}$-module. Let $A,B\in \mathcal{L}(\mathscr E)$ and let $\overline{\mathscr R(A^*)}$ and $\overline{\mathscr R(B^*)}$ be  orthogonally complemented submodules of $\mathscr E$. Then the following assertions are equivalent:
\begin{itemize}
\item[(i)] The operator equation $AX=B$ has a solution.
\item[(ii)] $\mathscr R(B)\subseteq \mathscr R(A)$.
\item[(iii)] $ \mathscr R(A)\supseteq \mathscr R(BB^*)$ and $BB^*\leq \lambda AA^*$  for some scalar $\lambda>0$.
\end{itemize}
\end{theorem}
\begin{proof}
The equivalency of (i) and (ii) follows from Theorem \ref{thFang}.\\
(ii)$\Longrightarrow$ (iii). Since  $\mathscr R(B)\subseteq \mathscr R(A)$ and $\overline{\mathscr R(A^*)}$ is an orthogonally complemented submodule of $\mathscr E$, we conclude from Theorem \ref{thFang} that  there exists $D\in \mathcal{L}(\mathscr E)$ such that $AD=B$. This shows that $BB^*\leq AA^*$. Clearly, $\mathscr R(BB^*)\subseteq \mathscr R(B) \subseteq  \mathscr R(A)$.\\
(iii)$\Longrightarrow$(ii). The hypothesis  $\mathscr R(BB^*)\subseteq \mathscr R(A)$ and Theorem \ref{thFang} show that there is a  reduced solution $S\in \mathcal{L}(\mathscr E)$ such that $AS=BB^*$. Also define $D:\mathscr R(A^*)\to \mathscr E$ by
\[
D(A^*x)=B^*x.
\]
Since $BB^*\leq \lambda AA^*$, we have $\|B^*x\|\leq \lambda^{\frac{1}{2}}\|A^*x\|$. Therefore,  $D$ is a well-defined bounded linear operator.  Let $\tilde{D}$ be the extension of $D$ defined on the whole $\mathscr E$ such that
$\tilde{D}|_{\mathscr N(A)}=0$ and $\tilde{D}A^*=B^*$. We show that $\tilde{D}\in \mathcal{L}(\mathscr E)$. We claim that $\tilde{D}$ is adjointable and so $A\tilde{D}^*=B$, which completes the proof. Indeed, for any $x,y\in \mathscr E$, we have
\begin{align*}
\langle \tilde{D}A^*x, B^*y\rangle =\langle B^*x, B^*y\rangle
=\langle x, BB^*y\rangle
= \langle x, ASy\rangle
=\langle A^*x,Sy\rangle.
\end{align*}
Hence
\begin{equation}\label{eqort}
\langle \tilde{D}x,y\rangle=\langle x,Sy\rangle\,\qquad (x\in \overline{\mathscr R(A^*)},y\in \overline{\mathscr R(B^*)})\,.
\end{equation}

Now suppose that $x,y\in \mathscr E$. Then  $x=x_1+x_2$ and $y=y_1+y_2$, where $x_1\in \overline{\mathscr R(A^*)},x_2\in \mathscr N(A), y_1\in \overline{\mathscr R(B^*)}$, and $y_2\in \mathscr N(B)$. In this case, we have
\begin{align*}
\langle \tilde{D}x,y\rangle &=\langle \tilde{D}(x_1+x_2),y_1+y_2\rangle\\
&=\langle \tilde{D}x_1,y_1\rangle\qquad({\rm since} ~\overline{\mathscr R(\tilde{D})}\subset \overline{\mathscr R(B^*)} ~{\rm and}~y_2\in \mathscr N(B))\\
&=\langle x_1,Sy_1\rangle\qquad ({\rm by~\eqref{eqort}})\\
&=\langle S^*x,y_1\rangle\qquad ({\rm since ~S ~is ~the~reduced ~solution ~AX=BB^*})\\
&=\langle S^*x,P_{B^*}y\rangle\\
&=\langle x,SP_{B^*}y\rangle\,.
\end{align*}
This shows that $\tilde{D}^*=SP_{B^*}$.
\end{proof}
In the following example, we show that there are operators $A,C\in \mathcal{L}(\mathscr E)$ such that $\overline{\mathscr R(A^*)}$ is orthogonally complemented, $\overline{\mathscr R(C^*)}$  is not orthogonally complemented, and $CC^*\leq \lambda AA^*$ for some positive scalar $\lambda$ but $\mathscr R(C)\not \subseteq \mathscr R(A)$; see \cite[Example 2.2]{MR}.
\begin{example}\label{ex1}
Let $\mathscr  A=C[0, 1]$ and let $\mathscr M=\{f\in \mathscr A, f(0) =0\}$ be its maximal ideal regarded as a Hilbert $\mathscr A$-module. Consider the operators  $A \in \mathcal L(\mathscr M, \mathscr A)$ and $C\in \mathcal L(\mathscr A)$ defined by
\[
(Af)(\lambda) = \lambda f(\lambda)\qquad (f\in \mathscr  M)
\]
and
\[
(Cf)(\lambda) = \lambda f(\lambda)\qquad (f\in\mathscr  A ).
\]
Set $\mathscr E=\mathscr A\oplus \mathscr M$. Let  $\tilde{A},\tilde{C}\in \mathcal L(\mathscr E)$ be defined by the block matrices of operators as follows:
\[
\tilde{A}=\begin{pmatrix}
0&A\\
0&0
\end{pmatrix}\\,\qquad\tilde{C}=\begin{pmatrix}
C&0\\
0&0
\end{pmatrix}\,.
\]
It is easy to see that $C^*=C$ and $(A^*f)(\lambda)=\lambda f(\lambda)$ for all $f\in \mathscr A$. Moreover,  $\tilde{C}\tilde{C}^*= \tilde{A}\tilde{A}^*$ and $\mathscr R(\tilde{C})\not \subseteq \mathscr R(\tilde{A})$, since for $f_0(\lambda) =\lambda $, we have 
\[\begin{pmatrix}
f_0\\0
\end{pmatrix} \in \mathscr R(\tilde{C}) \backslash \mathscr R(\tilde{A})\,.
\]
\end{example}
In the following example, we show that there are operators $A,B\in \mathcal{L}(\mathscr E)$ such that $\overline{\mathscr R(A^*)}$ and $\overline{\mathscr R(B^*)}$  are  orthogonally complemented and $\mathscr R (BB^*)\subseteq \mathscr R(A)$ but $\mathscr R(B)\not \subseteq \mathscr R(A)$.
\begin{example}
With the same notation as in Example \ref{ex1}, consider the operators  $A,D \in \mathcal L(\mathscr M, \mathscr A)$ defined by
\[
(Af)(\lambda) = \lambda f(\lambda)\,,\quad (Df)(\lambda)=\lambda^{\frac{2}{3}} f(\lambda)\qquad (f\in\mathscr  M ).
\]
One can verify that
\[
(A^*f)(\lambda) = \lambda f(\lambda)\,,\quad (D^*f)(\lambda)=\lambda^{\frac{2}{3}} f(\lambda)\qquad (f\in\mathscr  A).
\]
Set $\mathscr E=\mathscr A\oplus \mathscr M$. Let  $\tilde{A},\tilde{D}\in \mathcal L(\mathscr E)$ be defined  as follows:
\[
\tilde{A}=\begin{pmatrix}
0&A\\
0&0
\end{pmatrix}\\,\qquad\tilde{D}=\begin{pmatrix}
0&D\\
0&0
\end{pmatrix}\,.
\]
 Note that 
 for every $f\in \mathscr A$, we have 
\[
(DD^*f)(\lambda)=\lambda^{\frac{4}{3}}f(\lambda)=\lambda(\lambda^{\frac{1}{3}}f(\lambda))=(Ag)(\lambda)\,,
\]
where $g(\lambda)=\lambda^{\frac{1}{3}}f(\lambda)$ for all $\lambda\in[0,1]$. This  shows that $\mathscr R(\tilde{D}\tilde{D}^*)\subseteq \mathscr R(\tilde{A})$.
 In fact, for every $\begin{pmatrix}
f\\f'
\end{pmatrix}\in \mathscr E$, we have 
\begin{align*}
\tilde{D}\tilde{D}^*\begin{pmatrix}
f\\f'
\end{pmatrix}=\begin{pmatrix}
DD^*f\\0
\end{pmatrix}=\begin{pmatrix}
AA^*g\\0
\end{pmatrix}=\tilde{A}\begin{pmatrix}g\\0
\end{pmatrix}\,.
\end{align*}
Moreover,  $\mathscr R(\tilde{D})\not \subseteq \mathscr R(\tilde{A})$. Indeed, define $f\in \mathscr M$ by
\[
f(\lambda)=\lambda^{\frac{1}{3}}\,\qquad(\lambda\in [0,1])\,.
\]
Then 
\[
\tilde{D}(f)(\lambda)=1\qquad(\lambda\in [0,1])\,.
\]
This shows that $\tilde{D}f\not\in \mathscr R(\tilde{A})$.
\end{example}
For the next result, we need the following lemma.
\begin{lemma}\label{l2}
Let $(\mathscr H,\pi)$ be a representation of $\mathcal{L}(\mathscr E)$. Let $P$ be a projection on $\overline{\mathscr R(T)}$. Then $\pi(P)$ is the projection such that  $\overline{\mathscr{R}(\pi(T))}\subseteq \mathscr R(\pi(P))$.
\end{lemma}
\begin{proof}
Since $P$ is a projection, we have 
\[
\left(\pi(P)\right)^2=\pi(p)\pi(P)=\pi(P^2)=\pi(P)\,,\qquad(\pi(P))^*=\pi(P^*)=\pi(P)\,,
\]
so $\pi(P)$ is a projection. Also $PT=T$ gives that 
\[
\pi(P)\pi(T)=\pi(PT)=\pi(T)\,.
\]
Hence $\overline{\mathscr R(T)}\subseteq \mathscr R(\pi(P))$.
\end{proof}
\begin{theorem}\label{thvon}
Let $\mathcal{L}(\mathscr E)$ be a $W^*$-algebra. Let $A,B\in \mathcal{L}(\mathscr E)$ and let $\overline{\mathscr R(A^*)}=\mathscr E$. Then $BB^*\leq\lambda AA^*$ for some $\lambda>0$ if and only if $\mathscr R (B)\subseteq \mathscr R(A)$.
\end{theorem}
\begin{proof}
If $\mathscr R(B)\subseteq \mathscr R(A)$, then  Theorem \ref{thFang} entails that $BB^*\leq \lambda AA^*$ for some $\lambda>0$. For the reverse, let $(\mathscr H,\pi)$ be a faithful representation of $\mathcal{L}(\mathscr E)$. Since  $\overline{\mathscr R(A^*)}=\mathscr E$, we have $\mathscr N(A)=\{0\}$. From Lemma \ref{l2} applied to the projection $A^\dagger A$ onto $\overline{\mathscr R(A^*)}$, we deduce that $\overline{\mathscr R(\pi(A)^*)}=\mathscr H $, so $\mathscr N(\pi(A))=\{0\}$.  Now  we have
\[
\pi(B)\pi(B)^*=\pi(BB^*)\leq \pi(AA^*)\leq \pi(A)\pi(A)^*\,.
\]
Making use of \cite[Theorem 1]{Dou}, we obtain an operator  $X\in \mathcal{B}(\mathscr H)$ such  that
\begin{equation}\label{eq6}
\pi(A)X=\pi(B)\,.
\end{equation}   Given every $Z$ in the commutant  $\pi(\mathcal{L}(\mathscr E))'$ of $\pi(\mathcal{L}(\mathscr E))$, we have
\begin{align*}
\pi(A)ZX=Z\pi(A)X
=Z\pi(B)
=\pi(B)Z
=\pi(A)XZ.
\end{align*}
Since $\mathscr N(\pi(A))=\{0\}$, we have $XZ=ZX$. This gives that $X\in \pi(\mathcal{L}(\mathscr E))''=\pi(\mathcal{L}(\mathscr E))$. Hence there is $S\in \mathcal{L}(\mathscr E)$ such that $\pi(S)=X$. From \eqref{eq6}, we get
\begin{align*}
\pi(AS)=\pi(A)\pi(S)=\pi(A)X=\pi(B)\,.
\end{align*}
Since $(\mathscr E,\pi)$ is faithful, we conclude that $AS=B$ as desired.
\end{proof}
\begin{corollary}
Let $\mathscr A$ be a $W^*$-algebra. Let $\mathscr E$ be a self-dual Hilbert $C^*$-module.  Let $A,B\in \mathcal{L}(\mathscr E)$ and let $\overline{\mathscr R(A^*)}=\mathscr E$. Then the following assertions are equivalent:
\begin{itemize}
\item[(i)] $\mathscr R(B)\subseteq \mathscr R(A)$.
\item[(ii)] $BB^*\leq \lambda AA^*$ for some positive scalar $\lambda$.
\end{itemize}
\end{corollary}
\begin{proof}
It follows from \cite[Proposition 3.10]{Pa}  that $\mathcal{L}(\mathscr E)$ is a $W^*$-algebra. Hence Theorem \ref{thvon} gives the result.
\end{proof}

Set
\[
\mathscr N_{\mathscr E,f}=\{x\in \mathscr E; f(\langle x,x\rangle)=0\}\,,\quad {\rm and}\quad \mathscr N_{\mathscr A,f}=\{a\in \mathscr A; f(a^*a)=0\}.
\]
It is easy to see that $\mathscr N_{\mathscr E,f}$ is a closed submodule of $\mathscr E$ and that $\mathscr N_{\mathscr A,f}$ is a closed ideal of $\mathscr A$.  The next result reads as follows.
\begin{theorem}\label{thl2}
Let $\mathscr E={l}_2(\mathscr A)$ and let  $A,B\in \mathcal{L}(\mathscr E)$. A necessary condition for
\begin{equation}\label{eqineq}
BB^*\leq \lambda AA^*
\end{equation}
  is  that
 for every pure state $f$ on $\mathscr A$,
\begin{equation}\label{eqs}
 \mathscr R(B)\subseteq \mathscr R(A)+\mathscr N_{\mathscr E,f}
\end{equation}
be valid.
\end{theorem}
\begin{proof}
It is known from \cite[Theorem 5.2.4]{Mur}  that $\frac{\mathscr A}{\mathscr N_{\mathscr A,f}}$ is a Hilbert space endowed with the inner product defined by
\[
(a+\mathscr N_{\mathscr A,f},b+\mathscr N_{\mathscr A,f})_f=f(b^*a)\,.
\]
In addition, $\frac{\mathscr E}{\mathscr N_{\mathscr E,f}}$ is a pre-Hilbert space equipped with the inner product $\langle \cdot, \cdot\rangle_f$ defined by
\[
\langle x+\mathscr N_{\mathscr E,f},y+\mathscr N_{\mathscr E,f}\rangle_f=f(\langle x,y\rangle)\qquad (x,y\in \mathscr E).
\]
Next, we show that  $\frac{\mathscr E}{\mathscr N_{\mathscr E,f}}$ is complete. 

Define the $\mathscr A$-linear operator $\psi $ by
\[
\psi :\frac{\mathscr E}{\mathscr N_{\mathscr E,f}}\to l_2\left( \frac{\mathscr A}{\mathscr N_{\mathscr A,f}}\right)\,, \psi((a_i)+\mathscr N_{\mathscr E,f})=(a_i+\mathscr N_{\mathscr A,f}),
\]
If  $(a_i)\in \mathscr N_{\mathscr E,f}$, then $\sum_{i=1}^\infty f(a_i^*a_i)=0$ and  so $f(a_i^*a_i)=0$. Hence $a_i\in \mathscr N_{\mathscr A,f}$ for all $i\geq 1$.  Thus, $\psi $ is well-defined. Moreover, we have
\begin{align*}
\|\psi ((a_i)+\mathscr N_{\mathscr E,f})\|^2 &=\|(a_i+\mathscr N_{\mathscr A,f})\|^2 =\langle (a_i+\mathscr N_{\mathscr A,f}),(a_i+\mathscr N_{\mathscr A,f})\rangle\\ &=\sum_{i=1}^\infty f(a_i^*a_i)=\|(a_i)+\mathscr N_{\mathscr E,f}\|^2.
\end{align*}
Hence $\psi$ is an isomorphism ensuring  $\frac{\mathscr E}{\mathscr N_{\mathscr E,f}}$ be complete.

Now, define $\pi$ by
\[
\pi:\mathcal{L}(\mathscr E)\to \mathcal{B}\left(\frac{\mathscr E}{\mathscr N_{\mathscr E,f}}\right)\,,\pi(T)(x+\mathscr N_{\mathscr E,f})=Tx+\mathscr N_{\mathscr E,f}\,.
\]
We claim  that $\pi $ is a representation of $\mathcal{L}(\mathscr E)$. In fact, suppose that $T\in \mathcal{L}(\mathscr E)$. If $f(\langle x,x\rangle)=0$ for some $x\in \mathscr E$, then
\begin{align*}
f(\langle Tx,Tx\rangle)\leq \|T\|^2f(\langle x,x\rangle)=0\,.
\end{align*}
This shows $\pi(T)$ is a well-defined operator on $\frac{\mathscr E}{\mathscr N_{\mathscr E,f}}$. If $T,S\in \mathcal{L}(\mathscr E)$, then, for every $x\in \mathscr E$, we have
\begin{align*}
\pi(TS)(x+\mathscr N_{\mathscr E,f}) =TSx+\mathscr N_{\mathscr E,f}
=\pi(T)(Sx+\mathscr N_{\mathscr E,f})
=\pi(T)\pi(S)(x+\mathscr N_{\mathscr E,f}).
\end{align*}
It follows that $\pi(TS)=\pi(T)\pi(S)$. Moreover,
\begin{align*}
\langle\pi(T^*)(x+\mathscr N_{\mathscr E,f}),y+\mathscr N_{\mathscr E,f}\rangle &=\langle T^*x+\mathscr N_{\mathscr E,f},y+\mathscr N_{\mathscr E,f}\rangle=f(T^*x,y)=f(x,Ty)\\
&=\langle x+\mathscr N_{\mathscr E,f},Ty+\mathscr N_{\mathscr E,f}\rangle=\langle x+\mathscr N_{\mathscr E,f},\pi(T)(y+\mathscr N_{\mathscr E,f})\rangle  \\
&=\langle \pi(T)^*(x+\mathscr N_{\mathscr E,f}),y+\mathscr N_{\mathscr E,f}\rangle.
\end{align*}
Therefore, $\pi(T^*)=\pi(T)^*$. Hence $\pi$ is a representation of $\mathcal{L}(\mathscr E)$. \\
 Let $x\in \mathscr E$ be arbitrary.
Since $BB^*\leq \lambda AA^*$, we have $\pi(B)\pi(B)^*\leq \lambda \pi(A)\pi(A)^*$.  Theorem \ref{Douglas} yields $\mathscr R(\pi(B))\subseteq \mathscr R(\pi(A))$. Then there is $y\in \mathscr E$ such that
\[
\pi(B)(x+\mathscr N_{\mathscr E,f})=\pi(A)(y+\mathscr N_{\mathscr E,f})\,,
\]
whence
\[
Bx-Ay\in \mathscr N_{\mathscr E,f}\,.
\]
 Hence  $\mathscr R(B)\subseteq \mathscr R(A)+\mathscr N_{\mathscr E,f}$.
 \end{proof}
 The following example shows that the reverse of Theorem \ref{thl2} does not hold.
 \begin{example}
 Let $\mathscr A=\{f\in C[0,1]: f(0)=0\}$ and let $\mathscr E=l_2(\mathscr A)$. Define $A,B\in \mathcal{L}(\mathscr E)$ by
\[
A(f_1,f_2,f_3,\ldots)=(gf_1,0,0,\ldots) \quad \mbox{and}\quad B(f_1,f_2,f_3,\ldots)=(f_1,0,0,\ldots),\quad  (f_n)\in \mathscr E\,,
\]
where $g(\lambda)=\lambda$ for all $\lambda\in [0,1]$.
Let $\tau$ be any pure state on $\mathscr A$. There is $x_0>0$ such that $\mathscr N_{\mathscr A,\tau}=\{f\in \mathscr A: f(x_0)=0\}$. For every $(f_n)\in \mathscr E$, define $(h_n)\in \mathscr N_{\mathscr E,f}$ by
\[
h_1(\lambda)=\begin{cases}f_1(\lambda),&\lambda\leq\frac{x_0}{2},\\
\frac{2f_1(\frac{x_0}{2})}{x_0}(\lambda-x_0), &\frac{x_0}{2}\leq \lambda\leq x_0,\\
0,&\lambda\geq x_0,\end{cases}\,\quad \mbox{and} \quad h_n=0\quad(n\neq 1)\,.
\]
We also  define $(g_n)\in \mathscr E$ by  
\[
g_1(\lambda)=\begin{cases}\frac{f_1(\lambda)-h_1(\lambda)}{\lambda},&\lambda>0,\\0,&\lambda=0,\end{cases}\,\quad \mbox{and}\quad g_n=0\quad(n\neq 0)\,.
\]
We have
\[
B(f_n)=A(g_n)+(h_n)\,.
\]
We claim that $BB^*\not \leq\lambda AA^*$ for any $\lambda>0$. To prove this, we set 
\[
f_1(\lambda)=\lambda\qquad(\lambda\in [0,1])\,,
\]
and $f=(f_1,0,0,\ldots)$
to get
\[
(\langle BB^*f,f\rangle)(\lambda)=\lambda^2\not\leq \lambda^3=(\langle AA^*f,f\rangle)(\lambda)\,.
\]
 \end{example}

\bigskip

\noindent \textit{Conflict of Interest Statement.} On behalf of all authors, the corresponding author states that there is no conflict of interest.\\

\noindent\textit{Data Availability Statement.}  Data sharing is not applicable to this article as no datasets were generated or analyzed during the current study.

\bigskip


\bibliographystyle{amsplain}

\end{document}